\documentclass{amsart}

\usepackage{amssymb}
\usepackage[curve]{xypic}
\usepackage{fullpage}
\usepackage{epsfig}
\usepackage[mathcal]{euscript}
\usepackage{stmaryrd}

\newtheorem{theorem}{Theorem}

\newtheorem{prop}[theorem]{Proposition}

\newcommand{\qqed}{\qed \\[1ex]}
\newcommand{\co}[1]{{\langle {#1}\rangle}}

\newcommand{\Susp}{\Sigma}

\newcommand{\lra}{\longrightarrow}
\newcommand{\lla}{\longleftarrow}

\newcommand{\llra}[1]{\stackrel{#1}{\longrightarrow}}

\newcommand{\tensor}{\otimes}

\newcommand{\iso}{\;\cong\;}

\newcommand{\modmod}{/\!\!/ }

\mathchardef\hy="2D

\newcommand{\Mdef}[2]{\newcommand{#1}{\relax \ifmmode #2 \else $#2$\fi}}

\Mdef{\cA}{\mathcal{A}}
\Mdef{\cF}{\mathcal{F}}
\Mdef{\cP}{\mathcal{P}}
\Mdef{\cAll}{\mathcal{ALL}}
\Mdef{\RR}{\mathbf{R}}
\Mdef{\C}{\mathbf{C}}
\Mdef{\Z}{\mathbf{Z}}
\Mdef{\F}{\mathbf{F}}
\Mdef{\HF}{\mathbf{HF}}
\Mdef{\eps}{\varepsilon}

\begin{document}

\title{ On the Postnikov towers for 
real and complex 
connective 
K-theory\\
 }

\author{Robert R. Bruner}
\address{Department of Mathematics\\
	  Wayne State University\\
	  Detroit, Michigan  48202\\
	  USA
	 }
\email{rrb@math.wayne.edu}


\maketitle

\section{Introduction}

The analysis of real connective K-theory is facilitated by the
`$\eta c R$' cofiber sequence 
\[
\Sigma ko \llra{\eta} ko \llra{c} ku \llra{R} \Sigma^2 ko
\]
relating real and complex K-theories \cite{kobg}.    
Here we extend this relationship through
the Postnikov towers, producing several useful $ko$-module maps in the
process.

\begin{theorem}
\label{main}
The $\eta c R$ sequence lifts to cofiber sequences relating the connective covers
of $ko$ and $ku$ as follows:
\[
\xymatrix{
\Sigma ko
\ar^{\eta}[r]
&
ko
\ar^{c}[r]
&
ku
\ar^{R}[r]
&
\Sigma^2 ko
\\
\Sigma ko
\ar^{\eta_1}[r]
\ar@{=}[u]
&
ko\co{1}
\ar^{c_1}[r]
\ar[u]
&
\Sigma^2 ku
\ar^{r}[r]
\ar_{v}[u]
&
\Sigma^2 ko
\ar@{=}[u]
\\
\Sigma ko\co{1}
\ar^{\eta_2}[r]
\ar[u]
&
ko\co{2}
\ar^{c_2}[r]
\ar[u]
&
\Sigma^4 ku
\ar^{r_1}[r]
\ar_{v}[u]
&
\Sigma^2 ko\co{1}
\ar[u]
\\
\Sigma ko\co{2}
\ar^{\eta_4}[r]
\ar[u]
&
ko\co{4}
\ar^{c_4}[r]
\ar[u]
&
\Sigma^4 ku
\ar^{r_2}[r]
\ar@{=}[u]
&
\Sigma^2 ko\co{2}
\ar[u]
\\
\Sigma ko\co{4}
\ar^{\eta_8}[r]
\ar[u]
&
ko\co{8}
\ar^{c_8}[r]
\ar[u]
&
\Sigma^6 ku
\ar^{r_4}[r]
\ar_{v}[u]
&
\Sigma^2 ko\co{4}
\ar[u]
\\
\Sigma ko\co{8}
\ar^{\Sigma^8\eta}[r]
\ar[u]
&
ko\co{8}
\ar^{\Sigma^8 c}[r]
\ar@{=}[u]
&
\Sigma^8 ku
\ar^{\Sigma^8 R}[r]
\ar_{v}[u]
&
\Sigma^2 ko\co{8}
\ar[u]
\\
}
\]
\end{theorem}

In the sequence above, $c$ is complexification, $r$ is realification,
and $\eta$ is multiplication by $\eta \in ko_1$.  The map $R$
is an extension of realification $r$ over the Bott map: $r=Rv$.

We will write $X\co{n} \lra X$ for the
$n$-connected cover of $X$.  By this we mean that $\pi_i X\co{n} = 0$ for $i < n$, while
$\pi_i X\co{n} \lra \pi_i X$ is an isomorphism for $i \geq n$.
It will be useful to record the maps induced in cohomology.  All the modules
and maps we will deal with 
are in the image of induction from $\cA(1)\hy{\mathrm{Mod}}$,
\[
\cA \tensor_{\cA(1)} - : \cA(1)\hy{\mathrm{Mod}} \lra \cA\hy{\mathrm{Mod}},
\]
so we will record the results
in $\cA(1)\hy{\mathrm{Mod}}$, leaving it to the reader 
to tensor up.

The first lift, $\eta_1 c_1 r$, was brought to my attention by Vic Snaith
(\cite{snaith}).  The remaining lifts appeared at one point to be useful in
Geoffrey Powell's analysis of $ko^*BV_+$ (\cite{powell}), but in the end were unnecessary
there.

\section{Complex Periodicity}

In the complex case, periodicity and the Postnikov tower
amount to the same thing.
If we write $ku_* = \Z[v]$, with $|v|=2$,
then the Postnikov covers of $ku$
are simply given by multiplication by powers of $v$.

\medskip

\begin{minipage}{0.4\textwidth}
\[
\xymatrix
{
\Sigma^2 ku
\ar^{\simeq}[d]
\ar^{v}[dr]
&
\\
ku\co{2}
\ar[r]
&
ku
\\
}
\]
\end{minipage}
and more generally
\begin{minipage}{0.4\textwidth}
\[
\xymatrix
{
\Sigma^{2i+2} ku
\ar^{\simeq}[d]
\ar^{v}[r]
&
\Sigma^{2i} ku
\ar^{\simeq}[d]
\\
ku\co{2i+2}
\ar[r]
&
ku\co{2i}
\\
}
\]
\end{minipage}\\[2ex]

\begin{prop}
$ku \lra H\Z \lra \Susp^3 ku$ induces the short exact sequence
\[
\xymatrix{
\cA(1)/(Sq^1,Sq^3)
&&
\cA(1)/(Sq^1) 
\ar[ll]
&&
\Sigma^3
\cA(1)/(Sq^1,Sq^3) 
\ar[ll]
\\
}
\]
\begin{center}
\scalebox{0.7}{
$
\xymatrix{
\circ
\ar@/^1pc/[dd]
&&
\circ
\ar@/^1pc/[dd]
\ar[ll]
&&
\\
\\
\circ
&&
\circ
\ar[d]
&&
\\
&&
\circ
\ar@/^1pc/[dd]
&&
\circ
\ar@/^1pc/[dd]
\ar[ll]
\\
\\
&&
\circ
&&
\circ
\\
{\phantom{\cA(1)/(Sq^1,Sq^3)}}
&&
{\phantom{\cA(1)/(Sq^1) }}
&&
{\phantom{\Sigma^3 \cA(1)/(Sq^1,Sq^3) }}
\\
}
$
}
\end{center}
\end{prop}

\section{Real Periodicity}

In the real case, periodicity is broken into 4 steps.  We write
$ko_* =  \Z[\eta,\alpha,\beta]/(2\eta,\eta^3,\eta\alpha,\alpha^2-4\beta)$
with $|\eta|=1$, $|\alpha|=4$, and $|\beta| = 8$.

\[
\xymatrix{
\Sigma^8 ko
\ar^{\simeq}[r]
\ar_{\beta}[rdddd]
&
ko\co{8}
\ar[d]
\\
&
ko\co{4}
\ar[r]
\ar[d]
&
\Sigma^4 H\Z
\\
&
ko\co{2}
\ar[r]
\ar[d]
&
\Sigma^2 H\F_2
\\
&
ko\co{1}
\ar[r]
\ar[d]
&
\Sigma H\F_2
\\
&
ko
\ar[r]
&
H\Z
}
\]

\pagebreak

The following Proposition is well known.  It is a simple way to show that
a spectrum whose cohomology is $\cA\modmod \cA(1)$ must have $2$-local homotopy
additively isomorphic to $\pi_* ko$.

\begin{prop}
The maps induced in cohomology by  the Postnikov tower for $ko$ are as follows.
\begin{enumerate}
\item
\[
\xymatrix{
ko 
\ar[rr]
&&
H\Z 
\ar[rr]
&&
\Sigma ko\co{1}
}
\]
induces the short exact sequence
\[
\xymatrix{
&&
\F_2
&&
\cA(1)/(Sq^1) 
\ar[ll]
&&
\Sigma^2 \cA(1)/(Sq^2)
\ar[ll]
\\
}
\]
\begin{center}
\scalebox{0.7}{
$
\xymatrix{
\circ
&&&
\circ
\ar@/^1pc/[dd]
\ar[lll]
&&&
\\
\\
&&&
\circ
\ar[d]
&&&
\circ
\ar[d]
\ar[lll]
\\
&&&
\circ
\ar@/^1pc/[dd]
&&&
\circ
\ar@/^1pc/[dd]
\\
\\
&&&
\circ
&&&
\circ
\\
{\phantom{\F_2}}
&&
{\phantom{\cA(1)/(Sq^1)}}
&&
{\phantom{\Sigma^2 \cA(1)/(Sq^2)}}
\\
}
$
}
\end{center}
\item
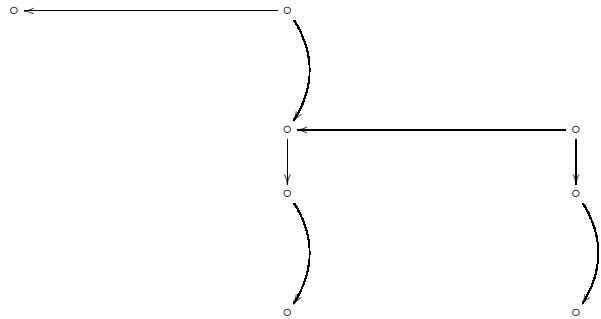
\[
\xymatrix{
ko\co{1} 
\ar[rrr]
&&&
\Sigma H\F_2 
\ar[rrr]
&&&
\Sigma ko\co{2}
}
\]
induces the short exact sequence
\[
\xymatrix{
\Sigma\cA(1)/(Sq^2)
&&
\Sigma \cA(1) 
\ar[ll]
&&
\Sigma(Sq^2) \iso \Sigma^3\cA(1)/(Sq^3)
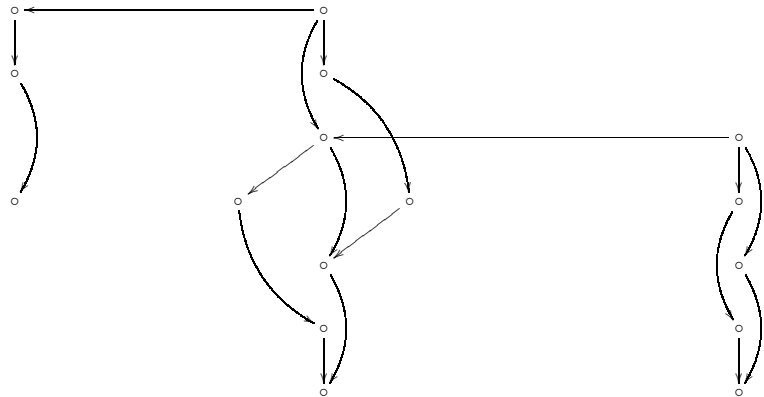
\ar[ll]
\\
}
\]
\begin{center}
\scalebox{0.7}{
$
\xymatrix{
\circ
\ar[d]
&&&&
\circ
\ar[llll]
\ar[d]
\ar@/_1pc/[dd]
&&&&&
\\
\circ
\ar@/^1pc/[dd]
&&&&
\circ
\ar@/^1pc/[ddr]
&&&&&
\\
&&&&
\circ
\ar@/^1pc/[dd]
\ar[dl]
&&&&&
\circ
\ar[lllll]
\ar[d]
\ar@/^1pc/[dd]
\\
\circ
&&&
\circ
\ar@/_1pc/[ddr]
&&
\circ
\ar[dl]
&&&&
\circ
\ar@/_1pc/[dd]
\\
&&&&
\circ
\ar@/^1pc/[dd]
&&&&&
\circ
\ar@/^1pc/[dd]
\\
&&&&
\circ
\ar[d]
&&&&&
\circ
\ar[d]
\\
&&&&
\circ
&&&&&
\circ
\\
{\phantom{\Sigma\cA(1)/(Sq^2)}}
&&&&
{\phantom{\Sigma \cA(1) }}
&&&&&
{\phantom{\Sigma(Sq^2) \iso \Sigma^3\cA(1)/(Sq^3)}}
\\
}
$
}
\end{center}
\pagebreak
\item
\[
\xymatrix{
ko\co{2} 
\ar[rrr]
&&&
\Sigma^2 H\F_2 
\ar[rrr]
&&&
\Sigma ko\co{4}
}
\]
induces the short exact sequence
\[
\xymatrix{
\Sigma^2\cA(1)/(Sq^3)
&&
\Sigma^2 \cA(1) 
\ar[ll]
&
\Sigma^2(Sq^3) \iso \Sigma^5\cA(1)/(Sq^1,Sq^2Sq^3)
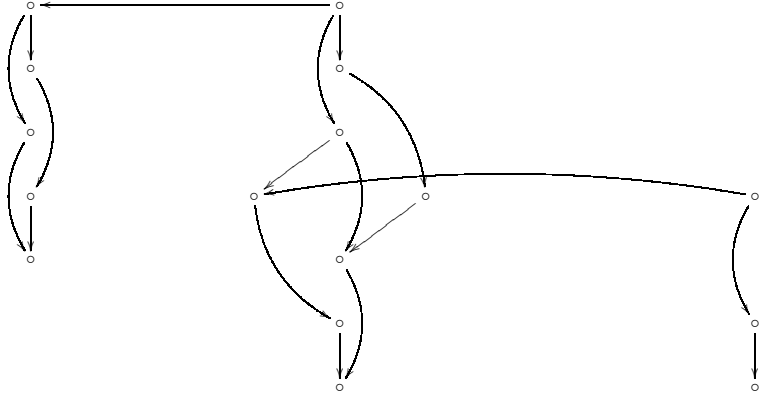
\ar[l]
\\
}
\]
\begin{center}
\scalebox{0.7}{
$
\xymatrix{
\circ
\ar[d]
\ar@/_1pc/[dd]
&&&&
\circ
\ar[llll]
\ar[d]
\ar@/_1pc/[dd]
&&&&&
\\
\circ
\ar@/^1pc/[dd]
&&&&
\circ
\ar@/^1pc/[ddr]
&&&&&
\\
\circ
\ar@/_1pc/[dd]
&&&&
\circ
\ar@/^1pc/[dd]
\ar[dl]
&&&&&
\\
\circ
\ar[d]
&&&
\circ
\ar@/_1pc/[ddr]
&&
\circ
\ar[dl]
&&&&
\circ
\ar@/_1pc/[llllll]
\ar@/_1pc/[dd]
\\
\circ
&&&&
\circ
\ar@/^1pc/[dd]
&&&&&
\\
&&&&
\circ
\ar[d]
&&&&&
\circ
\ar[d]
\\
&&&&
\circ
&&&&&
\circ
\\
{\phantom{\Sigma\cA(1)/(Sq^2)}}
&&&&
{\phantom{\Sigma \cA(1) }}
&&&&&
{\phantom{\Sigma(Sq^2) \iso \Sigma^3\cA(1)/(Sq^3)}}
\\
}
$
}
\end{center}
\item
\[
\xymatrix{
ko\co{4} 
\ar[rr]
&&
\Sigma^4 H\Z 
\ar[rr]
&&
\Sigma ko\co{8}
}
\]
induces the short exact sequence
\[
\xymatrix{
&&
\Sigma^4\cA(1)/(Sq^1,Sq^2Sq^3)
&
\cA(1)/(Sq^1) 
\ar[l]
&&
\Sigma^9 \F_2
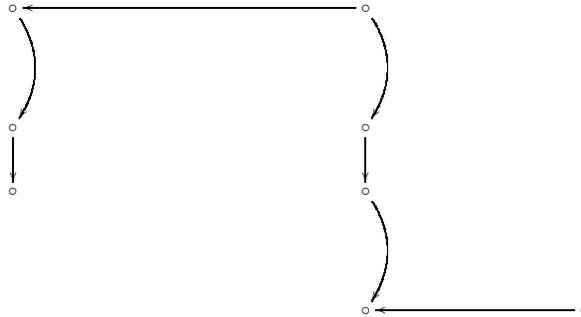
\ar[ll]
\\
}
\]
\begin{center}
\scalebox{0.7}{
$
\xymatrix{
\circ
\ar@/^1pc/[dd]
&&&
\circ
\ar@/^1pc/[dd]
\ar[lll]
&&&
\\
\\
\circ
\ar[d]
&&&
\circ
\ar[d]
&&&
\\
\circ
&&&
\circ
\ar@/^1pc/[dd]
&&&
\\
\\
&&&
\circ
&&&
\circ
\ar[lll]
\\
{\phantom{\Sigma^4\cA(1)/(Sq^1,Sq^2Sq^3)}}
&&
{\phantom{ \cA(1)/(Sq^1) }}
&&
{\phantom{\Sigma^8 \F_2}}
\\
}
$
}
\end{center}

\end{enumerate}
\end{prop}

\section{Maps of Postnikov towers}

First, we record the maps induced in cohomology by our starting point,
the $\eta c R$ sequence.

\begin{prop}
$ ko \llra{c} ku \llra{R} \Sigma^2 ko$ induces the short exact sequence
\[
\cA(1)/(Sq^1,Sq^2) 
 \lla \cA(1)/(Sq^1,Sq^3) \lla \Sigma^2 \cA(1)/(Sq^1,Sq^2) 
\]
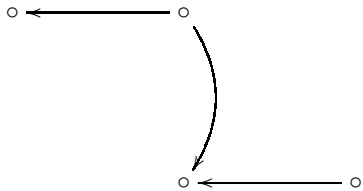
\[
\xymatrix{
\circ
&&
\circ
\ar@/^1pc/[dd]
\ar[ll]
&&
\\
\\
&&
\circ
&&
\circ
\ar[ll]
\\
}
\]
\end{prop}

We will now prove Theorem~\ref{main} in a series of steps.
We start with the braid of cofibrations induced by the composite 
$ko \llra{c} ku \lra H\Z$.

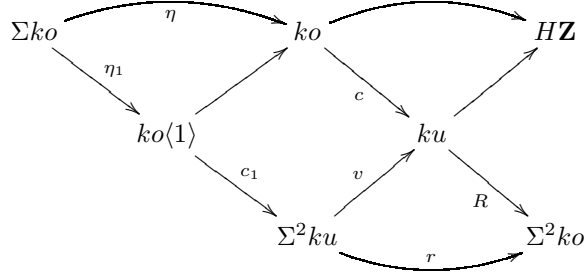
\[
\xymatrix
{
\Sigma ko
\ar@/^1pc/_{\eta}[rr]
\ar^{\eta_1}[rd]
&&
ko
\ar@/^1pc/[rr]
\ar_{c}[rd]
&&
H\Z
\\
&
ko\co{1}
\ar[ru]
\ar^{c_1}[rd]
&&
ku
\ar[ru]
\ar_{R}[rd]
&
\\
&&
\Sigma^2 ku
\ar@/_1pc/^{r}[rr]
\ar^{v}[ru]
&&
\Sigma^2 ko
&
\\
}
\]
This gives the $\eta_1 c_1 r$ sequence.  
To continue to the next step,
we will need to know the maps induced in cohomology by this one.

\begin{prop}
$ \Sigma ko \llra{\eta_1} ko\co{1} \llra{c_1} \Sigma^2 ku$ 
induces the short exact sequence
\[
\xymatrix{
\Sigma \F_2
&&
\Sigma \cA(1)/(Sq^2)
\ar[ll]
&&
\Sigma^2 \cA(1)/(Sq^1,Sq^3)
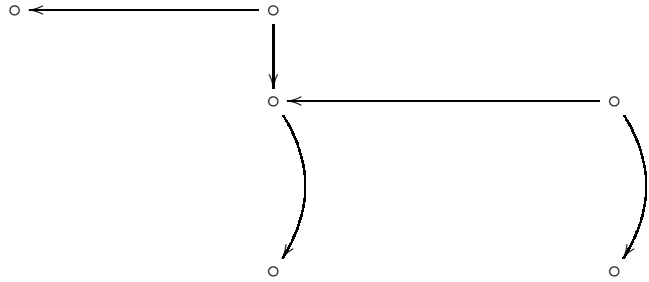
\ar[ll]
\\
\circ
&&
\circ
\ar[d]
\ar[ll]
&&
\\
&&
\circ
\ar@/^1pc/[dd]
&&
\circ
\ar[ll]
\ar@/^1pc/[dd]
\\
\\
&&
\circ
&&
\circ
\\
}
\]
\end{prop}

\begin{proof}
These are the only maps which can make the long exact sequence exact.
\end{proof}

From this we observe that we have a commutative square
\[
\xymatrix{
\Sigma H\F_2
\ar^{Sq^1}[r]
&
\Sigma^2 H\Z
\\
ko\co{1}
\ar^{c_1}[r]
\ar[u]
&
\Sigma^2 ku
\ar[u]
\\
}
\]
which induces the following map of cofiber sequences.  
The map induced in cohomology by ${\eta_1}$ implies that
the left hand map $\Sigma ko \lra \Sigma H\Z$ is nontrivial.
This implies that
the fiber of $c_2$ is $\Sigma ko\co{1}$,
giving the next Postnikov lift
of the $\eta c R$ sequence.

\[
\xymatrix{
\Sigma H\Z
\ar[r]
&
\Sigma H\F_2
\ar^{Sq^1}[r]
&
\Sigma^2 H\Z
\ar[r]
&
\Sigma^2 H\Z
\\
\Sigma ko
\ar[u]
\ar^{\eta_1}[r]
&
ko\co{1}
\ar[u]
\ar^{c_1}[r]
&
\Sigma^2 ku
\ar[u]
\ar^{r}[r]
&
\Sigma^2 ko
\ar[u]
\\
\Sigma ko\co{1}
\ar[u]
\ar^{\eta_2}[r]
&
ko\co{2}
\ar[u]
\ar^{c_2}[r]
&
\Sigma^4 ku
\ar[u]
\ar^{r_1}[r]
&
\Sigma^2 ko\co{1}
\ar[u]
\\
}
\]

Again
we need to record the maps induced in cohomology for use in the next step.
\pagebreak

\begin{prop}
$ \Sigma ko\co{1} \llra{\eta_2} ko\co{2} 
\llra{c_2} \Sigma^4 ku$ 
induces the short exact sequence
\[
\xymatrix{
\Sigma^2 \cA(1)/(Sq^2)
&
\Sigma^2 \cA(1)/(Sq^3)
\ar^{\eta_2^*}[l]
&
\Sigma^4 \cA(1)/(Sq^1,Sq^3)
\ar^{c_2^*}[l]
\\
}
\]
\begin{center}
\scalebox{0.7}{
$
\xymatrix{
\circ
\ar[d]
&&
\circ
\ar[d]
\ar@/_1pc/[dd]
\ar[ll]
&&
\\
\circ
\ar@/^1pc/[dd]
&&
\circ
\ar@/^1pc/[dd]
&&
\\
&&
\circ
\ar@/_1pc/[dd]
&&
\circ
\ar@/_1pc/[dd]
\ar[ll]
\\
\circ
&&
\circ
\ar[d]
&&
\\
&&
\circ
&&
\circ
\\
{\phantom{\Sigma^2 \cA(1)/(Sq^2)}}
&&
{\phantom{\Sigma^2 \cA(1)/(Sq^3)}}
&&
{\phantom{\Sigma^4 \cA(1)/(Sq^1,Sq^3)}}
\\
}
$
}
\end{center}

\end{prop}

Now consider the braid of cofibrations induced by the composite 
$ko\co{4} \lra ko\co{2} \lra \Sigma^4 ku $.

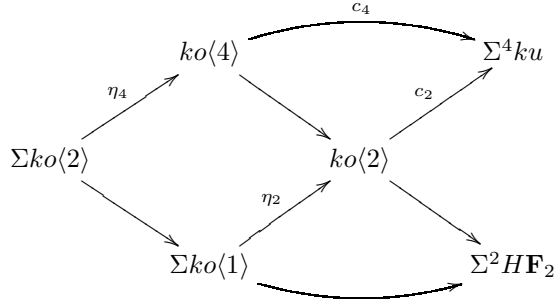
\[
\xymatrix
{
&&
ko\co{4}
\ar@/^1pc/^{c_4}[rr]
\ar[rd]
&&
\Sigma^4 ku
\\
&
\Sigma ko\co{2}
\ar^{\eta_4}[ru]
\ar[rd]
&&
ko\co{2}
\ar^{c_2}[ru]
\ar[rd]
&
\\
&&
\Sigma ko\co{1}
\ar@/_1pc/[rr]
\ar^{\eta_2}[ru]
&&
\Sigma^2 H\F_2
&
\\
\\
}
\]

Since $\eta_2^*$  is nonzero in degree 2, the map
$\Sigma ko\co{1} \lra \Sigma^2 H\F_2$  is nontrivial, and hence the
the fiber of $c_4$ is $\Sigma ko\co{2}$.
Again, we need to record the maps induced in cohomology, and again, they 
`roll' one step to the left.

\begin{prop}
$ \Sigma^3 ku \llra{r_2} \Sigma ko\co{2} \llra{\eta_4} ko\co{4} $
induces the short exact sequence

\[
\xymatrix{
\Sigma^3 \cA(1)/(Sq^1,Sq^3)
&
\Sigma^3 \cA(1)/(Sq^3)
\ar^{r_2^*}[l]
&
\Sigma^4 \cA(1)/(Sq^1,Sq^2Sq^3)
\ar^{\eta_4^*}[l]
\\
}
\]
\begin{center}
\scalebox{0.7}{
$
\xymatrix{
\circ
\ar@/^1pc/[dd]
&&
\circ
\ar[d]
\ar@/^1pc/[dd]
\ar[ll]
&&
\\
&&
\circ
\ar@/_1pc/[dd]
&&
\circ
\ar@/_1pc/[dd]
\ar[ll]
\\
\circ
&&
\circ
\ar@/^1pc/[dd]
&&
\\
&&
\circ
\ar[d]
&&
\circ
\ar[d]
\\
&&
\circ
&&
\circ
\\
{\phantom{\Sigma^3 \cA(1)/(Sq^1,Sq^3)}}
&&
{\phantom{\Sigma^3 \cA(1)/(Sq^3)}}
&&
{\phantom{\Sigma^4 \cA(1)/(Sq^1,Sq^2Sq^3)}}
\\
}
$
}
\end{center}

\end{prop}

Since $\eta_4^*$ sends the generator to $Sq^1$, we get a map of cofiber sequences
whose fiber gives the next lift, $\eta_8 c_8 r_4$.

\[
\xymatrix{
\Susp^3 H\Z
\ar[r]
&
\Sigma^3 H\F_2
\ar_{Sq^1}[r]
&
\Susp^4 H\Z
\ar[r]
&
\Susp^4 H\Z
&
\\
\Sigma^3 ku
\ar[u]
\ar^{r_2}[r]
&
\Sigma ko\co{2}
\ar[u]
\ar^{\eta_4}[r]
&
ko\co{4}
\ar[u]
\ar^{c_4}[r]
&
\Sigma^4 ku
\ar[u]
\\
\Sigma^5 ku
\ar[u]
\ar^{r_4}[r]
&
\Sigma ko\co{4}
\ar[u]
\ar^{\eta_8}[r]
&
ko\co{8}
\ar[u]
\ar^{c_8}[r]
&
\Sigma^6 ku
\ar[u]
\\
\\
}
\]

\begin{prop}
$ \Sigma^5 ku \llra{r_4} \Sigma ko\co{4} \llra{\eta_8} ko\co{8} $
induces the short exact sequence

\[
\xymatrix{
\Sigma^5 \cA(1)/(Sq^1,Sq^3)
&
\Sigma^5 \cA(1)/(Sq^1, Sq^2Sq^3)
\ar^{r_4^*}[l]
&
\Sigma^8 \F_2
\ar^(.3){\eta_8^*}[l]
\\
}
\]
\begin{center}
\scalebox{0.7}{
$
\xymatrix{
\circ
\ar@/^1pc/[dd]
&&
\circ
\ar@/^1pc/[dd]
\ar[ll]
&
\\
&&
&
\\
\circ
&&
\circ
\ar[d]
&
\\
&&
\circ
&&&
\circ
\ar[lll]
\\
{\phantom{\Sigma^5 \cA(1)/(Sq^1,Sq^3)}}
&
{\phantom{\Sigma^5 \cA(1)/(Sq^1, Sq^2Sq^3)}}
&
{\phantom{\Sigma^8 \F_2}}
\\
}
$
}
\end{center}
\end{prop}

Finally, consider the braid of cofibrations induced by the composite
$\Sigma ko\co{8} \lra \Sigma ko\co{4} \llra{\eta_8} ko\co{8}$.

\[
\xymatrix{
\Sigma ko\co{8}
\ar@/^1pc/^{\widetilde{\eta}}[rr]
\ar[dr]
&&
ko\co{8}
\ar@/^1pc/^{c_8}[rr]
\ar^{\widetilde{c}}[dr]
&&
\Sigma^6 ku
\\
&
\Sigma ko\co{4}
\ar^{\eta_8}[ru]
\ar[rd]
&&
\Sigma^8 ku
\ar^{\widetilde{R}}[rd]
\ar^{v}[ru]
&&
\\
\Sigma^5 ku
\ar@/_1pc/[rr]
\ar^{r_4}[ru]
&&
\Sigma^5 H\Z
\ar[ru]
\ar@/_1pc/[rr]
&&
\Sigma^2 ko\co{8}
&
\\
\\
}
\]

Since $r_4^*$ is an isomorphism on $H^5$, the map $\Sigma^5 ku \lra \Sigma^5 H\Z$
must be the bottom cohomology generator, justifying the appearance of 
$\Sigma^8 ku$ and $v$ in this braid.

The result is
a cofiber sequence $\Sigma^9 ko \lra \Sigma^8 ko \lra \Sigma^8 ku$. 
The maps are $ko$-module maps by construction, and agree 
with the 8-fold suspensions of $\eta$, $c$ and $R$ in homotopy, by
the maps $X\co{8} \lra X$.  
The adjunction $F_{ko}(\Sigma^9 ko, \Sigma^8 ko) \simeq F(S^9, \Sigma^8 ko)$,
shows that a $ko$-module map $\Sigma^9 ko \lra \Sigma^8 ko$
is determined by its effect on homotopy.  
Therefore, the first map, and hence the other two,
are the 8-fold suspensions
of $\eta$, $c$ and $R$.
\qqed

\end{document}